\documentclass[12pt]{amsart}
\usepackage[pagewise]{lineno}
\usepackage{amscd,amssymb}
\usepackage[arrow,matrix]{xy}
\usepackage[plainpages,backref,urlcolor=blue]{hyperref}
\usepackage{mathrsfs}
\usepackage{amsthm}
\usepackage{bm}
\usepackage{amscd}
\usepackage[shortlabels]{enumitem}
\usepackage{mathdots}
\usepackage{mathrsfs}
\usepackage{graphics}
\usepackage{mathtools}

\theoremstyle{plain}
\newtheorem{thm}[subsection]{Theorem}
\newtheorem{lem}[subsection]{Lemma}
\newtheorem{prop}[subsection]{Proposition}
\newtheorem{cor}[subsection]{Corollary}
\newtheorem{claim}[subsection]{Claim}

\theoremstyle{definition}
\newtheorem{rk}[subsection]{Remark}

\newcommand{\ov}{\overline}
\newcommand{\bb}{\mathbb}

\newcommand{\coker}{\text{Coker}}

\newcommand{\p}{\partial}

\newcommand{\Mod}[1]{\,\text{\rm mod}\,#1}

\newcommand{\Hom}{\text{\rm Hom}}

\begin{document}
	\date{}
		
	\title[Deformations of Nodal Hypersurfaces]{On Deformations of Nodal Hypersurfaces }
		
	\author[ZHENJIAN WANG]{ ZHENJIAN WANG  }
	\address{Univ. C$\hat{\rm o}$te d'Azur, CNRS,  LJAD, UMR 7351, 06100 Nice, France.}
	\email{wzhj01@gmail.com}
		
		
		
	\begin{abstract}
	We extend the infinitesimal Torelli theorem for smooth hypersurfaces to nodal hypersurfaces.
	\end{abstract}
	\maketitle

\section{Introduction}
Deformations of smooth hypersurfaces provide examples of great interest and importance in the theory of variation of Hodge structures, especially because of the generic Torelli theorem, see \cite{VO2}, Chapter 6. In a recent thesis \cite{Zhao}, Y. Zhao considers  deformations of nodal surfaces in the $3$-dimensional complex projective space $\bb{P}^3$ and shows that the infinitesimal Torelli theorem still holds.

Let $S=\bb{C}[x_0,\cdots, x_n]=\bigoplus_{d=0}^\infty S_d$ be the graded ring of polynomials and let $f\in S_d$ be a homogeneous polynomial of degree $d$. Denote by $X_f: f=0$ the hypersurface in $\bb{P}^n$ defined by $f$. Moreover, let
$$
J(f)=(\frac{\p f}{\p x_0},\cdots, \frac{\p f}{\p x_n})
$$
 be the graded ideal generated by the first derivatives of $f$, also called the Jacobian ideal of $f$. We consider the following map
\begin{equation}\label{eq: phi}
\varphi:\quad (S/J(f))_d\to \Hom((S/J(f))_{d-n-1}, (S/J(f))_{2d-n-1}),\qquad [P]\mapsto([Q]\mapsto [PQ]).
\end{equation}
As a matter of fact, Y. Zhao \cite{Zhao} proves the infinitesimal Torelli theorem by showing that the map $\varphi$
is injective when $n=3$ and $X_f$ is a {\bf nodal} surface. This result can be extended to higher dimensional cases.

\begin{thm}\label{main thm}
Assume $n\geq 3$ is an integer and $d\geq n+1$. Let $f\in S$ be a homogeneous polynomial of degree $d$ such that $X_f: f=0$ is a nodal hypersurface in $\bb{P}^n$. Then the map $\varphi$ is injective.
\end{thm}

As it is proved in \cite{Zhao}, Chapter 3, Example 3.1.3, $(S/J(f))_d$ parameterizes the equivalence classes of deformations of the pair $(\bb{P}^n, X_f)$. Alternatively, let $GL=GL(n+1,\bb{C})$ be the general linear group of rank $n+1$. Then $GL$ acts on $S_d$ by coordinate transformations and for any $f\in S_d$, the tangent space at $f$ of the orbit $GL\cdot f$ is given by $J(f)_d$, see \cite{D87}, Chapter 4, Formula (4.16). It follows that $(S/J(f))_d$ can be seen as the set of directions in $S_d$ that are transversal to the orbit $GL\cdot f$  at $f$. In addition, any smooth analytic subset $\mathcal{U}\subseteq S_d$ can be seen as a family of hypersurfaces in $\bb{P}^n$. If $f\in\mathcal{U}$ and $T_f\mathcal{U}\cap J(f)_d=\{0\}$, then we call $\mathcal{U}$ an {\bf effective} deformation of $f$. From this point of view, $(S/J(f))_d$ is the maximal set of effective deformations of $f$.

Now let $X_f: f=0$ be a nodal hypersurface in $\bb{P}^n$ and let $n(f)$ be the number of nodes in $X_f$. Then we have a moduli space, denoted by $\mathfrak{B}_f\subseteq S_d$, parameterizing all nodal hypersurfaces in $\bb{P}^n$ having exactly $n(f)$ nodes. By the discussion following Corollary 3.8 in \cite{D92}, Chapter 1, we have that $\mathfrak{B}_f$ is a constructible subvariety of $S_d$ and the topological type of $(\bb{P}^n, X_g)$ is locally trivial for $g\in \mathfrak{B}_f$. Moreover, for any $g$ lying in the connected component of $\mathfrak{B}_f$ containing $f$, $(\bb{P}^n, X_g)$ is topologically equivalent to $(\bb{P}^n, X_f)$.

Now assume $\mathcal{U}\subseteq\mathfrak{B}_f$ is a connected {\bf smooth} subvariety and $f\in\mathcal{U}$. For any $g\in\mathcal{U}$, $X_g$ is homeomorphic to $X_f$ by the local topological triviality of the pair $(\bb{P}^n,X_g)$. So there is a natural identification $H^{n-1}_0(X_g)\cong H^{n-1}_0(X_f)$, where $H^{n-1}_0(X_g)$ is the primitive cohomology of $X_g$ defined by $H^{n-1}_0(X_g)=\coker(H^{n-1}(\bb{P}^n)\to H^{n-1}(X_g))$. In particular, $\dim H^{n-1}_0(X_g)$ is constant for $g\in\mathcal{U}$.

Moreover, $H^{n-1}_0(X_g)$ has a natural mixed Hodge structure, since $X_g$ is a singular algebraic variety, see \cite{PS}, Part II, Chapter 5. It turns out that $\dim F^{n-1}H^{n-1}_0(X_g)$ and $\dim F^{n-2}H^{n-1}_0(X_g)$ are also constant for $g\in\mathcal{U}$ (in most cases), see Corollary \ref{cor: dimension} below. Thus, we have the following well-defined map
\begin{equation}\label{eq: mP}
\mathcal{P}:\qquad \mathcal{U}\ni g\mapsto (F^{n-1}H^{n-1}_0(X_g), F^{n-2}H^{n-1}_0(X_g))\in\mathcal{F}
\end{equation}
where $\mathcal{F}$ is the corresponding flag manifold of subspaces of $H^{n-1}_0(X_f)$.

By relating the primitive cohomology with the graded pieces of the algebra $S/J(f)$ and applying Theorem \ref{main thm}, we prove the following, as a generalization of \cite{Zhao}, Chapter 3.
\begin{thm}\label{main cor}
Assume $n\geq 3$ is odd or $n\geq 6$ is even. Let $X_f: f=0$ be a nodal hypersurface in $\bb{P}^n$ of degree $d\geq n+1$ and let $\mathcal{U}\subseteq \mathfrak{B}_f$ be a smooth subvariety of $\mathfrak{B}_f$ which gives an effective deformation of $X_f$. Then the map $\mathcal{P}$ above is well-defined and the differential $d\mathcal{P}$ is injective at $f$.
\end{thm}

Thus, loosely speaking, the infinitesimal Torelli theorem also holds for nodal hypersurfaces.

Note that for smooth hypersurfaces, the {\bf generic} Torelli theorem holds, see \cite{VO2}, Part II, Chapter 6, Section 6.3.2, and it remains an interesting question whether this is also the case for nodal hypersurfaces. Recall that in the proof of the generic Torelli theorem for smooth hypersurfaces, the essential part is to show that a generic homogeneous polynomial can be reconstructed from its Jacobian ideal, which also holds for nodal hypersurfaces by Theorem 1.1 in \cite{ZW}, because a generic $f$ of degree $d>3$ with the associated hypersurface $X_f$ having a fixed number of nodes is not of Sabastiani-Thom type, which is the only exception for $f$ not to be reconstructed from $J(f)$; another key ingredient in the smooth case is the symmetriser lemma, which is still open for nodal hypersurfaces.\\

The author would like to thank an anonymous referee, whose remarks make the exposition of this paper improved.

\section{Syzygies of the Jacobian ideal}\label{sec: syzygies}
Let $K^\bullet(f)$ be the Koszul complex of $\frac{\p f}{\p x_0},\cdots, \frac{\p f}{\p x_n}$ with the natural grading $\deg(x_j)=1$ and $\deg(dx_j)=1$:
$$
K^\bullet(f):\quad 0\to \Omega^0\to\Omega^1\to\cdots\to\Omega^{n+1}\to 0
$$
where $\Omega^1=\sum_{i=0}^n Sdx_i$ and $\Omega^p=\bigwedge^p\Omega^1$, and the differentials are given by the wedge product with $df=\sum_{i=0}^n\frac{\p f}{\p x_i}dx_i$.

The homogeneous component of the cohomology group $H^n(K^\bullet(f))_{n+r}$ describes the syzygies
$$
\sum_{j=0}^n a_j \frac{\p f}{\p x_j}=0
$$
with $a_j\in S_r$ modulo the trivial syzygies generated by
$$
\frac{\p f}{\p x_i}\frac{\p f}{\p x_j}+\biggl(-\frac{\p f}{\p x_i}\biggr)\biggl(\frac{\p f}{\p x_j}\biggr)=0,\qquad i<j.
$$
We may restate the main result in \cite{D13} or Theorem 9 in \cite{DS14} in the following form.
\begin{lem}\label{lem: syzygies}
Let $X_f: f=0$ be a nodal hypersurface in $\bb{P}^n$ of degree $d>2$ and $n\geq 3$, then
$H^n(K^\bullet(f))_m=0$ for any
$$
m\leq \frac{nd-1}{2}.
$$
\end{lem}

Let $f_s\in S_d$ be such that $X_{f_s}: f_s=0$ is a smooth hypersurface. It is well-known that $\dim (S/J(f_s))_k$ depends only on $n,d$ and $k$, see \cite{D87}, Chapter 7, Proposition 7.22. In the introduction part of \cite{D13R}, the following two notions are given:
$$
ct(X_f)=\max\{q\quad:\quad \dim (S/J(f))_k=\dim (S/J(f_s))_k\text{ for all }k\leq q\}
$$
and
$$
mdr(X_f)=\min\{q\quad:\quad H^n(K^\bullet(f))_{q+n}\neq 0\}.
$$
They have the following relation
$$
ct(X_f)=mdr(X_f)+d-2,
$$
see loc. cit.. We have the following.

\begin{lem}\label{lem: dimension}
Let $X_f: f=0$ be a nodal hypersurface in $\bb{P}^n$ of degree $d\geq n+1$ and $n\geq 3$, then
$$
\dim (S/J(f))_k=\dim (S/J(f_s))_k,\quad k=d-n-1, 2d-n-1.
$$
In particular, $\dim (S/J(f))_k$ does not depend on the concrete equation of the polynomial $f$ for $k=d-n-1, 2d-n-1$.
\end{lem}

\begin{proof}
We only need to check that $2d-n-1\leq ct(X_f)$. Indeed, by Lemma \ref{lem: syzygies}, we immediately have
$$
ct(X_f)=mdr(X_f)+d-2\geq\biggl(\frac{nd-1}{2}-n\biggr)+d-2>2d-n-1,
$$
where the last inequality follows from $n\geq 3$ and $d\geq n+1$.
\end{proof}

\subsection{Proof of Theorem \ref{main thm}}
To prove Theorem \ref{main thm}, we first prove the following.
\begin{lem}\label{lem: in}
Assume $X_f: f=0$ is a nodal hypersurface in $\bb{P}^n$ of degree $d\geq n+1$ and $n\geq 3$.
Let $G\in S_t$ such that $t<2d-n-1$ and $Gx_j\in J(f)$ for all $j=0,\cdots, n$, then $G\in J(f)$.
\end{lem}
\begin{proof}
Assume
\begin{equation}\label{eq: 1}
Gx_i=\sum_{k=0}^n H_{ik}\frac{\p f}{\p x_k},\quad i=0,\cdots, n,
\end{equation}
with $H_{ik}\in S_{t+2-d}, i,k=0,\cdots, n$, then
$$
0=x_i(x_jG)-x_j(x_iG)=\sum_{k=0}^n(x_iH_{jk}-x_jH_{ik})\frac{\p f}{\p x_k}.
$$
Note that
$$
t+3-d+n\leq(2d-n-2)+3+n-d=d+1\leq\frac{nd-1}{2},
$$
so by Lemma \ref{lem: syzygies}, we get $x_iH_{jk}-x_jH_{ik}\in J(f)$ for all $i,j,k=0,\cdots, n$ while all these polynomials have degree $t+3-d<(2d-n-1)+3-d=d-n+2\leq d-1$, so they must all vanish identically; in particular,
$$
x_iH_{jk}-x_jH_{ik}=0,\quad i\neq j
$$
thus, $x_i|H_{ik}$. It follows that $G\in J(f)$ as desired.
\end{proof}

{\it Proof of Theorem \ref{main thm}: }
We first remark that Theorem \ref{main thm} holds when $d=n+1$. In fact, in this case, $J(f)_{d-n-1}=J(f)_0=0$ and $(S/J(f))_{d-n-1}=S_0=\bb{C}$ consists of constants. Since $1\in (S/J(f))_{d-n-1}$ and $\varphi([P])(1)=[P]$, one sees easily that $\varphi$ is injective.

Thus, in the sequel of the proof, we will focus on the case $d>n+1$.

Aiming at a contradiction, we assume that there exists $P\in S_d\setminus J(f)_d$ such that $\varphi([P])=0$.

Then there exists a $Q\in S_l, 0\leq l<d-n-1$ such that $PQ\notin J(f)$ and $l$ is chosen to be maximal. By the maximality of $l$, we have $(PQ)x_j\in J(f)$ for all $j=0,\cdots, n$. Note that $PQ\in S_{l+d}$ and $l+d<2d-n-1$, hence by Lemma \ref{lem: in}, $PQ\in J(f)$, contradiction.

\section{Hodge theory for nodal hypersurfaces}\label{sec: hodge}
Let $X_f: f=0$ be a nodal hypersurface in $\bb{P}^n$ of degree $d\geq n+1$ and $n\geq 3$. The cohomology groups under consideration below all have $\bb{C}$ as coefficients unless otherwise explicitly pointed out.

By the Lefschetz hyperplane theorem for singular varieties (see \cite{Ha}), we have
$$
H^i(X_f)=H^i(\bb{P}^n),\qquad i<n-1,
$$
and
$$
H^{n-1}(\bb{P}^n)\to H^{n-1}(X_f)
$$
is injective. Let
$$
H^{n-1}_0(X_f)=\coker(H^{n-1}(\bb{P}^n)\to H^{n-1}(X_f)),
$$
be the primitive cohomology of $X_f$. Then $H^{n-1}_0(X_f)$ admits a mixed Hodge structure. Moreover, let $U_f=\bb{P}^n\setminus X_f$ be the complement of $X=X_f$, then $H^n(U_f)$ also admits a mixed Hodge structure and $H^n(U_f)$ and $H^{n-1}_0(X_f)$ are closely related.

\subsection{Relation between $H^*(U_f)$ and $H^*(X_f)$}

Let $X_f^*$ be the smooth locus of $X_f$ and let
$$
H^{n-1}_0(X_f^*)=\coker(H^{n-1}(\bb{P}^n)\to H^{n-1}(X_f^*)).
$$
Then $H^{n-1}_0(X_f^*)$ has a natural mixed Hodge structure. Moreover, as is shown in \cite{D92}, Chapter 6, Corollary 3.11, there is a natural residue isomorphism
\begin{equation}\label{eq: UX}
\ov{R}_f:\quad H^n(U_f)\xrightarrow{\sim} H^{n-1}_0(X_f^*)
\end{equation}
which is also an isomorphism of mixed Hodge structures of type $(-1,-1)$.

Let $i: X_f^*\to X_f$ be the inclusion. We have the naturally induced homomorphisms in cohomology
$$
i^*: H^{n-1}(X_f)\to H^{n-1}(X_f^*)
$$
and
$$
i^*_0: H^{n-1}_0(X_f)\to H^{n-1}_0(X_f^*).
$$
Moreover, $i^*, i^*_0$ are also morphisms of mixed Hodge structures. Our discussion will be divided into two cases, regarding whether $n$ is odd or even.

\subsubsection{Case 1: $n$ is odd }
When $n$ is odd, the variety $X_f$ is a $\bb{Q}$-homology manifold, i.e., for any point $x\in X_f$, $H^i(X_f, X_f\setminus\{x\},\bb{Q})=\bb{Q}$ if $i=2n$ and $0$ otherwise. Moreover, we have the following claim.

\begin{claim}
$i^*$ and $i_0^*$ are both isomorphisms.
\end{claim}
\begin{proof}
Indeed, we have a long exact sequence of mixed Hodge structures with respect to the pair $(X_f,X_f^*)$:
\begin{equation}\label{eq: revise1}
\rightarrow H^{n-1}(X_f,X_f^*)\to H^{n-1}(X_f)\xrightarrow{i^*} H^{n-1}(X_f^*)\to H^{n}(X_f,X_f^*)
\end{equation}
Let $x_i, i=1,\cdots, r$ be all the nodes in $X_f$, then $X_f^*=X_f\setminus\{x_1,\cdots, x_r\}$, and furthermore, by the excision theorem
$$
H^{n-1}(X_f,X_f^*)=\bigoplus_{i=1}^rH^{n-1}(X_f,X_f\setminus\{x_i\})=0,
$$
since $X_f$ is a $\bb{Q}$-homology manifold and $n-1\neq 0,2n$ for $n\geq 3$. Similarly, $H^n(X_f,X_f^*)=0$. Thus, it follows from \eqref{eq: revise1} that $i^*$ and $i^*_0$ are both isomorphisms.
\end{proof}
Note that the weights of $H^{n-1}(X_f)$ are $\leq n-1$ since $X_f$ is compact while the weights of $H^{n-1}(X_f^*)$ are $\geq n-1$ since $X_f^*$ is smooth (see \cite{PS}, p. 131, Table 5.1), hence both $H^{n-1}(X_f^*)$ and $H^{n-1}(X_f)$ have pure Hodge structures of weight $n-1$ and it follows from the isomorphism \eqref{eq: UX} that $H^n(U_f)$ has a pure Hodge structure of weight $n+1$.

Let
$$
R_f=(i_0^*)^{-1}\circ\ov{R}_f:\quad H^n(U_f)\to H^{n-1}_0(X_f).
$$
Then $R_f$ is an isomorphism of mixed Hodge structures of type $(-1,-1)$.  It follows that we have isomorphisms
$$
R_f:\quad F^pH^n(U_f)\xrightarrow{\sim}F^{p-1}H^{n-1}_0(X_f)
$$
for all $p$. In particular, there are isomorphisms
\begin{equation}\label{hodge1}
R_f:\quad Gr_F^{p+1}H^n(U_f)\xrightarrow{\sim}Gr_F^pH^{n-1}_0(X_f),\quad p=n-1,n-2.
\end{equation}

\subsubsection{Case 2: $n$ is even }
When $n$ is even, $X_f$ is no longer a $\bb{Q}$-homology manifold. However, there is still an explicit description of the relations between $H^n(U_f)$ and $H^{n-1}_0(X_f)$. Note that in this case $H^{n-1}(\bb{P}^n)=0$ and thus
$$
H^{n-1}_0(X_f)=H^{n-1}(X_f),\quad H^{n-1}_0(X_f^*)=H^{n-1}(X_f^*)
$$
and $i^*=i^*_0$. Moreover, there exists an exact sequence of mixed Hodge structures
\begin{equation}\label{eq: exact}
\cdots\to H^{n-1}(X_f,X_f^*)\to H^{n-1}(X_f)\xrightarrow{i^*} H^{n-1}(X_f^*)\to H^n(X_f,X_f^*)\to\cdots.
\end{equation}

To make use of this exact sequence, we first give the following claim.
\begin{claim}
For $k=n-1,n$, $H^k(X_f,X_f^*)$ has a pure Hodge structure of type $(\rho_k,\rho_k)$ for some $\rho_k\in\bb{N}$.
\end{claim}
\begin{proof}
Let $a_1,\cdots, a_m$ be the nodes in $X_f$ and $B_i\ni a_i, i=1,\cdots,m$ be a small ball in $\bb{P}^n$ around $a_i$ such that $B_i\cap B_j=\emptyset$ for $i\neq j$.

By the excision theorem and conic structure theorem (see \cite{D92}, Chapter 1, Theorem 5.1),
$$
H^k(X_f,X_f^*)=\bigoplus_{i=1}^mH^k(B_i\cap X_f, B_i\cap X_f\setminus\{a_i\})\simeq\bigoplus_{i=1}^m H^{k-1}(K_i),\quad k=n-1,n
$$
where $K_i$ is the link of $X_f$ around $a_i$ ($i=1,\cdots, m$).

For each $i$, $K_i$ has the homotopy type of the unit sphere bundle of tangent bundle of $S^{n-1}$. Indeed, locally around $a_i$, $X_f$ is defined as $z_1^2+\cdots+z_n^2=0$, where $(z_1,\cdots, z_n)$ is the local coordinate system of $\bb{P}^n$ centered at $a_i$. Then $K_i$ can be described as
$$
K_i=\{(z_1,z_2,\cdots,z_n)\in\bb{C}^n\quad :\quad \sum_{j=1}^nz_j^2=0,\ \text{and}\ \sum_{j=1}^n|z_j|^2=\epsilon^2\ \}
$$
where $\epsilon>0$ is small. Let
$$
z_j=\frac{\epsilon}{\sqrt{2}}(v_j+\sqrt{-1}w_j),\quad  j=1,\cdots,n
$$
and
$$
v=(v_1,\cdots, v_n), w=(w_1,\cdots, w_n),
$$
then
$$
K_i=\{(v,w)\in\bb{R}^n\times\bb{R}^n\quad :\quad |v|^2=|w|^2=1\ \text{and}\ \langle v,w\rangle=0\ \},
$$
which is the unit sphere bundle of tangent bundle of $S^{n-1}$.

It follows that
$$
H^{k-1}(K_i)=\bb{C},\quad k=n-1,n.
$$
Note also that $H^{k-1}(K_i)=H^k(B_i\cap X_f,B_i\cap X_f\setminus\{a_i\})=H^k(X_f,X_f\setminus\{a_i\})$ admits a natural mixed Hodge structure. In particular,
$$
1=\dim H^{k-1}(K_i)=\sum_{w\in\bb{N}}\dim Gr^W_wH^{k-1}(K_i)=\sum_{w\in\bb{N}}\sum_{p+q=w}\dim (Gr^W_wH^{k-1}(K_i))^{p,q}
$$
where $Gr^W_wH^{k-1}(K_i)$ is a pure Hodge structure of weight $w$ and
$$
Gr^W_wH^{k-1}(K_i)=\bigoplus_{p+q=w}(Gr^W_wH^{k-1}(K_i))^{p,q}
$$
is the Hodge decomposition. By the Hodge symmetry, we have
$$
(Gr^W_wH^{k-1}(K_i))^{p,q}=\ov{(Gr^W_wH^{k-1}(K_i))^{q,p}}.
$$
It follows that there exists $\rho_{k,i}\in\bb{N}$ such that
$$
Gr^W_wH^{k-1}(K_i)=0,\qquad w\neq 2\rho_k
$$
and
$$
(Gr^W_{2\rho_{k,i}}H^{k-1}(K_i))^{p,q}=0,\qquad p\neq q
$$
and
$$
\dim (Gr^W_{2\rho_{k,i}}H^{k-1}(K_i))^{\rho_{k,i},\rho_{k,i}}=1.
$$
In particular, $H^{k-1}(K_i)$ is pure of type $(\rho_{k,i},\rho_{k,i})$.

Note that the mixed Hodge structure on $H^{k-1}(K_i)$ depends only on the local structure of $X_f$ around $a_i$ (see \cite{Df83}, Theorem 3.4). Since all the $a_i$'s are nodes, $H^{k-1}(K_i)$ is naturally isomorphic to $H^{k-1}(K_j)$ as mixed Hodge structures for any $i,j$, hence there exists $\rho_k\in\bb{N}$ such that
$$
\rho_{k,1}=\rho_{k,2}=\cdots=\rho_{k,m}=\rho_k,
$$
and thus $H^k(X_f, X_f^*)$ is pure of type $(\rho_k,\rho_k)$ for $k=n-1,n$.
\end{proof}

By Proposition (C28) in \cite{D92}, Appendix C (see also \cite{Df83}, Proposition 3.8) , it follows that $2\rho_{n-1}\leq n-2$. Thus,
$$
Gr_F^pH^{n-1}(X_f,X_f^*)=0,\quad p=n-2,n-1.
$$
Moreover, by the discussions above Example 3.18 in \cite{D92}, Chapter 6, $H^{n-1}(K_i)$ has weight $n$, namely, $2\rho_n=n$, and thus for $n\geq 6$
$$
Gr_F^pH^n(X_f,X_f^*)=0,\quad p=n-2,n-1.
$$

Therefore, it follows from \eqref{eq: exact} that we have an isomorphism
$$
i_0^*:\quad Gr_F^{n-1}H^{n-1}_0(X_f)=F^{n-1}H^{n-1}_0(X_f)\xrightarrow{\sim}Gr_F^{n-1}H^{n-1}_0(X_f^*)=F^{n-1}H^{n-1}_0(X_f^*)
$$
for $n\geq 4$. Furthermore, we have isomorphisms
$$
i_0^*:\quad Gr_F^{n-2}H^{n-1}_0(X_f)\xrightarrow{\sim}Gr_F^{n-2}H^{n-1}_0(X_f^*)
$$
and
$$
i_0^*:\quad F^{n-2}H^{n-1}_0(X_f)\xrightarrow{\sim}F^{n-2}H^{n-1}_0(X_f^*)
$$
for $n\geq 6$; but for $n=4$, we only have injections
$$
i_0^*:\quad Gr_F^{n-2}H^{n-1}_0(X_f)\hookrightarrow Gr_F^{n-2}H^{n-1}_0(X_f^*).
$$
and
$$
i_0^*:\quad F^{n-2}H^{n-1}_0(X_f)\hookrightarrow F^{n-2}H^{n-1}_0(X_f^*).
$$

Using the residue isomorphism \eqref{eq: UX}, we denote
$$
F^{n-1}(U_f,X_f)=\ov{R}_f^{-1}\biggl(i_0^*(F^{n-2}H^{n-1}_0(X_f))\biggr)\subseteq F^{n-1}H^n(U_f)
$$
for $n\geq 4$ (and $n$ is even). Then clearly, $F^{n-1}(U_f,X_f)=F^{n-1}H^n(U_f)$ for $n\geq 6$.

We still denote by $\ov{R}_f$ its restriction to $F^{n-1}(U_f,X_f)$. Then
$$
i_0^*:\quad F^{n-2}H^{n-1}_0(X_f)\to \ov{R}_f(F^{n-1}(U_f,X_f))
$$
is an isomorphism and we have an isomorphism
$$
R_f=(i_0^*)^{-1}\circ\ov{R}_f:\quad F^{n-1}(U_f,X_f)\xrightarrow{\sim} F^{n-2}H^{n-1}_0(X_f).
$$

\subsubsection{Conclusion }
In conclusion, no matter whether $n$ is even or odd, we always have isomorphisms
\begin{equation}\label{eq: hodgen-1}
R_f:\quad Gr_F^nH^n(U_f)\xrightarrow{\sim}Gr_F^{n-1}H^{n-1}_0(X_f)
\end{equation}
and
\begin{equation}\label{eq: hodgen-2}
R_f:\quad F^{n-1}(U_f,X_f)/F^nH^n(U_f)\xrightarrow{\sim}Gr_F^{n-2}H^{n-1}_0(X_f)
\end{equation}
where $F^{n-1}(U_f,X_f)=\ov{R}_f^{-1}\biggl(i_0^*(F^{n-2}H^{n-1}_0(X_f))\biggr)$ is a subspace of $F^{n-1}H^n(U_f)$ containing $F^nH^n(U_f)$; and $R_f=(i_0^*)^{-1}\circ\ov{R}_f$.

\subsection{Cohomology of $X_f$}

Denote by
$$
\Omega=\sum_{i=0}^n(-1)^ix_idx_0\wedge\cdots\wedge dx_{i-1}\wedge\widehat{dx_i}\wedge dx_{i+1}\wedge\cdots\wedge dx_n
$$
where $\widehat{(\  )}$ means that the term is omitted. As is shown in \cite{D92}, Chapter 6, any cohomology class in $F^pH^n(U_f)$ can be represented by a form
$$
\omega(h)=\frac{h\Omega}{f^{n-p+1}}
$$
with $h\in S_{(n-p+1)d-n-1}$. Hence, by \eqref{eq: hodgen-1}, we see that any element in $Gr_F^{n-1}H^{n-1}_0(X_f)$ can be represented by
$$
R_f([\frac{h_1\Omega}{f}])
$$
with $h_1\in S_{d-n-1}$ and similarly, by \eqref{eq: hodgen-2}, any element in $Gr_F^{n-2}H^{n-1}_0(X_f)$ can be represented by
$$
R_f([\frac{h_2\Omega}{f^2}])
$$
with $h_2\in S_{2d-n-1}$.

Such results agree with \cite{DSW}, Theorem 2.2, where the following formulae are given
$$
Gr_F^nH^n(U_f)=(S/J(f))_{d-n-1},\quad Gr_F^{n-1}H^n(U_f)=(S/J(f))_{2d-n-1},
$$
for $n>3$ and for $n=3$,
$$
Gr_F^nH^n(U_f)=(S/J(f))_{d-n-1},\quad Gr_F^{n-1}H^n(U_f)=(I(f)/J(f))_{2d-n-1},
$$
where $I(f)$ is the saturation of $J(f)$, which is also equal to the radical of $J(f)$ for a nodal hypersurface (see \cite{D13R}, Remark 2.2).

Putting all the discussions above in this section together, we obtain the following.
\begin{prop}\label{prop: hodge}
Let $X_f: f=0$ be a nodal hypersurface in $\bb{P}^n$ of degree $d\geq n+1$. Then
\begin{enumerate}[\rm(i)]
\item when $n\geq 3$, there is an isomorphism
$$
\Lambda_f:\quad (S/J(f))_{d-n-1}\to Gr_F^{n-1}H^{n-1}_0(X_f),\quad \Lambda_f(h_1)=R_f([\frac{h_1\Omega}{f}]),
$$
\item when $n>4$, there is an isomorphism
$$
\Lambda_f:\quad (S/J(f))_{2d-n-1}\to Gr_F^{n-2}H^{n-1}_0(X_f),\quad \Lambda_f(h_2)=R_f([\frac{h_2\Omega}{f^2}]),
$$
\item when $n=3$, there is an isomorphism
$$
\Lambda_f:\quad (I(f)/J(f))_{2d-n-1}\to Gr_F^{n-2}H^{n-1}_0(X_f),\quad \Lambda_f(h_2)=R_f([\frac{h_2\Omega}{f^2}]),
$$
\item when $n=4$, there is an isomorphism
$$
\Lambda_f:\quad S'/J(f)_{2d-n-1}\to Gr_F^{n-2}H^{n-1}_0(X_f),\quad \Lambda_f(h_2)=R_f([\frac{h_2\Omega}{f^2}]),
$$
where $S'\subseteq S_{2d-n-1}$ is a vector subspace containing $J(f)_{2d-n-1}$ obtained via
$$
S'/J(f)_{2d-n-1}=\omega^{-1}(F^{n-1}(U_f,X_f)/F^nH^n(U_f))
$$
where $\omega$ is the isomorphism
$$
\omega:\quad (S/J(f))_{2d-n-1}\to Gr_F^{n-1}H^n(U_f),\qquad \omega(h_2)=[\frac{h_2\Omega}{f^2}]
$$
established in \cite{DSW}, Theorem 2.2, and $F^{n-1}(U_f,X_f)$ is obtained in \eqref{eq: hodgen-2}.

\end{enumerate}
In all the formulae above, $R_f$ denotes the residue map.
\end{prop}

As a corollary, we have the following.

\begin{cor}\label{cor: dimension}
Let $X_f: f=0$ be a nodal hypersurface in $\bb{P}^n$ of degree $d\geq n+1$. Then
\begin{enumerate}[\rm(i)]
\item if $n\geq 3$, the dimension
$$
\dim F^{n-1}H^{n-1}_0(X_f)
$$
depends only on $n,d$.
\item if $n\geq 3$ is odd or $n\geq 6$ is even, the dimension
$$
\quad \dim F^{n-2}H^{n-1}_0(X_f)
$$
depends only on $n,d$ and possibly the number of nodes in $X_f$.
\end{enumerate}
\end{cor}
\begin{proof}
Note that
$$
\dim F^{n-1}H^{n-1}_0(X_f)=\dim Gr_F^{n-1}H^{n-1}_0(X_f)
$$
and
$$
\dim F^{n-2}H^{n-1}_0(X_f)=\dim Gr_F^{n-1}H^{n-1}_0(X_f)+\dim Gr_F^{n-2}H^{n-1}_0(X_f).
$$

If $n>4$, the results follow from Proposition \ref{prop: hodge} and Lemma \ref{lem: dimension}, and the dimensions depend only on $n,d$.
When $n=3$, $X_f$ is a $\bb{Q}$-homology manifold and the Hodge numbers of $X_f$ depend only on $n,d$ and the number of nodes in $X_f$, see also \cite{D96}.
\end{proof}

\section{Variations of mixed Hodge structures}\label{sec: var}

Let $X_f: f=0$ be a nodal hypersurface in $\bb{P}^n$ of degree $d\geq n+1$. When $n$ is odd, assume $n\geq 3$ while when $n$ is even, assume $n\geq 6$.

\subsection{Topological triviality}
Recall that $\mathfrak{B}_f\subseteq S_d$ parameterizes all nodal hypersurfaces with the same number of nodes as $X_f$. Let $\mathcal{U}\subseteq\mathfrak{B}_f$ be a contractible smooth subvariety containing $f$ such that it gives an effective deformation for $X_f$. Set
$$
\mathfrak{X}_{\mathcal{U}}=\{(x,g)\in\bb{P}^n\times\mathcal{U}\quad:\quad x\in X_g\quad \}
$$
which can be seen as the union of all nodal hypersurfaces parameterized by $\mathcal{U}$.

Then by the First Thom Isotopy Lemma (see \cite{D92}, Chapter 1, Section 3), there is a homeomorphism $\Phi$ satisfying the following commutative diagram
$$
\xymatrix{
(\bb{P}^n\times\mathcal{U},\mathfrak{X}_{\mathcal{U}})\ar[rr]^{\Phi}\ar[rd]_{p_1} & & (\bb{P}^n, X_f)\times\mathcal{U}\ar[ld]^{p_2}\\
& \mathcal{U} &
}
$$
where $p_1,p_2$ are natural projections. In fact, $\Phi$ can be obtained by integrating some well-controlled stratified vector field; for a proof, see \cite{Ma}. From now on, we fix such a homeomorphism.

In particular, for any $g\in\mathcal{U}$, there is a canonical homeomorphism $\Phi_g:\bb{P}^n\to\bb{P}^n$, which induces homeomorphisms $\Phi_{g,X}: X_f\to X_g$ and $\Phi_{g,U}: U_f\to U_g$ with $\Phi_f=\text{\rm Id}$.

Moreover, we have an induced isomorphism of groups
$$
\Phi_{g,X}^*:\quad H^{n-1}_0(X_g)\xrightarrow{\sim}H^{n-1}_0(X_f).
$$
Hence $\dim H^{n-1}_0(X_g)$ is constant for $g\in\mathcal{U}$.

In addition, by Corollary \ref{cor: dimension}, under our assumption on $n$, the dimensions
$$
\dim F^{n-1}H^{n-1}_0(X_g), \quad \dim F^{n-2}H^{n-1}_0(X_g)
$$
are constant with respect to $g\in\mathcal{U}$.
Via the identification $\Phi_{g, X}^*: H^{n-1}_0(X_g)\xrightarrow{\sim} H^{n-1}_0(X_f)$, it follows that $(F^{n-1}H^{n-1}_0(X_g), F^{n-2}H^{n-1}_0(X_g))$ can be identified with $(\Phi_{g,X}^*F^{n-1}H^{n-1}_0(X_g), \Phi_{g,X}^*F^{n-2}H^{n-1}_0(X_g))$, which are two subspaces of $H^{n-1}_0(X_f)$ of fixed dimension. Therefore, we have the well-defined map as in \eqref{eq: mP}
\begin{equation*}
\mathcal{P}:\qquad \mathcal{U}\ni g\mapsto (\Phi_{g,X}^*F^{n-1}H^{n-1}_0(X_g), \Phi_{g,X}^*F^{n-2}H^{n-1}_0(X_g))\in\mathcal{F}
\end{equation*}
where $\mathcal{F}$ is the following flag manifold
\begin{eqnarray*}
\mathcal{F}&=&\{(E_1,E_2)\quad:\quad E_1\subseteq E_2\text{ are vector subspaces of }H^{n-1}_0(X_f)\text{ and }\\
& &\dim E_1=\dim F^{n-1}H^{n-1}_0(X_f)\text{ and }\dim E_2=\dim F^{n-2}H^{n-1}_0(X_f)\}.
\end{eqnarray*}
When $n$ is odd, all the Hodge numbers of $X_g$ are constant for $g\in\mathcal{U}$, and $\mathcal{P}$ is just two components of the period map in the theory of variation of Hodge structures, see \cite{VO1}, Part III, Chapter 10.

\subsection{Infinitesimal deformation}
Now we consider the differential of $\mathcal{P}$. Note that a component of $d\mathcal{P}_f$ is the map
$$
d\mathcal{P}_f:\quad T_f\mathcal{U}\to \Hom(F^{n-1}H^{n-1}_0(X_f), H^{n-1}_0(X_f)/F^{n-1}H^{n-1}_0(X_f));
$$
for the properties of tangent spaces of flag manifolds, we refer to \cite{VO1}, Part III, Chapter 10 and for analogous treatments for smooth hypersurfaces, see \cite{VO2}, Part II, Chapter 6. Recall that Proposition \ref{prop: hodge} implies that any element in  $F^{n-1}H^{n-1}_0(X_f)$ is of the form
$$
\omega(h_1)=R_f([\frac{h_1\Omega}{f}]).
$$
The following holds.
\begin{lem}\label{lem: dp}
For $h\in T_f\mathcal{U}\subseteq S_d$, we have
$$
d\mathcal{P}_f(h)(\omega(h_1))=d\mathcal{P}_f(h)\biggl(R_f([\frac{h_1\Omega}{f}])\biggr)=R_f([-\frac{hh_1\Omega}{f^2}]) \Mod F^{n-1}H^{n-1}_0(X_f)
$$
\end{lem}

Its proof is a little lengthy and we postpone it to the end of this section; instead, we first derive Theorem \ref{main cor} from Lemma \ref{lem: dp}.

\subsection{Proof of Theorem \ref{main cor} }
From Lemma \ref{lem: dp} and Proposition \ref{prop: hodge}, the image of $d\mathcal{P}_f$ is contained in
\begin{align*}
&\Hom(F^{n-1}H^{n-1}_0(X_f), F^{n-2}H^{n-1}_0(X_f)/F^{n-1}H^{n-1}_0(X_f))\\
=&\Hom(Gr_F^{n-1}H^{n-1}_0(X_f), Gr_F^{n-2}H^{n-1}_0(X_f)).
\end{align*}
Moreover, we get the following commutative diagram
\begin{equation}\label{eq: square}
\xymatrix{
T_f\mathcal{U}\ar[r]^--{d\mathcal{P}_f}\ar[d]^{i_1}  & \Hom(Gr_F^{n-1}H^{n-1}_0(X_f), Gr_F^{n-2}H^{n-1}_0(X_f))\ar[d]^{i_2}\\
(S/J(f))_d\ar[r]^--{\varphi} & \Hom((S/J(f))_{d-n-1}, (S/J(f))_{2d-n-1})
}
\end{equation}
where $\varphi$ is given in \eqref{eq: phi}. $i_1$ is the composite $T_f\mathcal{U}\subseteq S_d\to S_d/J(f)_d$, which is injective since $\mathcal{U}$ is an effective deformation. $i_2$ is defined as follows: for $\eta\in\Hom(Gr_F^{n-1}H^{n-1}_0(X_f), Gr_F^{n-2}H^{n-1}_0(X_f))$ and $h_1\in(S/J(f))_{d-n-1}$,
$$
i_2(\eta)(h_1)=-\Lambda_f^{-1}\biggl(\eta\bigl(\Lambda_f(h_1)\bigr)\biggr),
$$
where $\Lambda_f$ is the isomorphism given in Proposition \ref{prop: hodge}.

By Theorem \ref{main thm}, $\varphi$ is injective, hence $\varphi\circ i_1$ is injective. Thus it follows from \eqref{eq: square} that $d\mathcal{P}_f$ is injective, hence Theorem \ref{main cor} follows.

\begin{rk}
The result is probably also true for $n=4$. We exclude this case because we do not know whether the dimension $\dim F^{n-2}H^{n-1}_0(X_g)$ or equivalently $\dim Gr_F^{n-2}H^{n-1}_0(X_g)$ is constant for $g\in\mathcal{U}$ in this case.
\end{rk}

\subsection{Proof of Lemma \ref{lem: dp}} The proof is almost the same as that in \cite{VO2}, Part II, Chapter 6 where variations of smooth hypersurfaces are considered. However, to avoid any possible confusion, we give the details here.

From the topological triviality of the family $X_g, g\in\mathcal{U}$, it follows that there exists a small contractible neighbourhood $\mathcal{N}\ni f$ in $\mathcal{U}$, such that for any $g\in \mathcal{N}$, $X_g$ is a deformation retract of
$$
\mathcal{X}_{\mathcal{N}}:=\bigcup_{g\in\mathcal{N}}X_g\subseteq\bb{P}^n.
$$
Set
$$
U_{\mathcal{N}}=\bb{P}^n\setminus\ X_{\mathcal{N}}.
$$
Then $U_{\mathcal{N}}$ is a deformation retract of $U_g$ for every $g\in\mathcal{N}$. Let for $g\in\mathcal{N}$
$$
\tau_g:\quad U_{\mathcal{N}}\hookrightarrow U_g
$$
be the natural inclusion, then the induced homomorphism in cohomology
$$
\tau_g^*:\quad H^n(U_g)\to H^n(U_{\mathcal{N}})
$$
is an isomorphism.

The differential $d\mathcal{P}_f$ can be computed as follows:
for any $h\in T_f\mathcal{U}\subseteq S_d$, choose a curve $g(t): (-\epsilon,\epsilon)\to\mathcal{N}\subseteq\mathcal{U}$ such that $g(0)=f$ and $\frac{d g}{dt}(0)=h$. For any element in $F^{n-1}H^{n-1}_0(X_f)$ of the form
$$
\omega(h_1)=R_f([\frac{h_1\Omega}{f}]),
$$
let
$$
\omega_t(h_1)=R_{g(t)}([\frac{h_1\Omega}{g(t)}])
$$
give an element of $F^{n-1}H^{n-1}_0(X_{g(t)})$. Then
$$
d\mathcal{P}_f(h)(\omega(h_1))=\frac{d}{d t}\biggl|_{t=0}\Phi_{g(t),X}^*(\omega_t(h_1))\quad\Mod F^{n-1}H^{n-1}_0(X_f).
$$
We have
\begin{eqnarray*}
\frac{d}{d t}\biggl|_{t=0} \Phi_{g(t),X}^*(\omega_t(h_1))&=&\frac{d}{dt}\biggl|_{t=0}R_f\biggl(\Phi_{g(t),U}^*([\frac{h_1\Omega}{g(t)}])\biggr)\\
  &=&R_f\biggl(\frac{d}{dt}\biggl|_{t=0}\Phi_{g(t),U}^*([\frac{h_1\Omega}{g(t)}])\biggr),
\end{eqnarray*}
where $\Phi_{g,U}^*$ is the homomorphism induced by the map $\Phi_{g,U}: U_f\to U_g$. Note that $\Phi_{g(t),U}^*: H^n(U_{g(t)})\to H^n(U_f)$ is equal to the composition
$$
H^{n}(U_{g(t)})\xrightarrow{\tau_{g(t)}^*}H^n(U_{\mathcal{N}})\xrightarrow{(\tau_f^*)^{-1}} H^n(U_f).
$$
Hence,
\begin{eqnarray*}
\frac{d}{d t}\biggl|_{t=0} \Phi_{g(t),X}^*(\omega_t(h_1))&=&R_f\biggl(\frac{d}{dt}\biggl|_{t=0}(\tau_{f}^*)^{-1}\tau_{g(t)}^*[\frac{h_1\Omega}{g(t)}]\biggr)\\
          &=&R_f\biggl((\tau_{f}^*)^{-1}\frac{d}{dt}\biggl|_{t=0}[\tau_{g(t)}^*\frac{h_1\Omega}{g(t)}]\biggr).\\
\end{eqnarray*}
Note that $\tau_{g(t)}^*$ acting on forms is a restriction map, it follows that
\begin{eqnarray*}
\frac{d}{dt}\biggl|_{t=0}[\tau_{g(t)}^*\frac{h_1\Omega}{g(t)}]&=&\frac{d}{dt}\biggl|_{t=0}[\frac{h_1\Omega}{g(t)}\biggl|_{U_\mathcal{N}}]\\
            &=&[\frac{d}{dt}\biggl|_{t=0}\frac{h_1\Omega}{g(t)}\biggl|_{U_\mathcal{N}}]\\
            &=&[-\frac{hh_1\Omega}{f^2}\biggl|_{U_\mathcal{N}}].
\end{eqnarray*}
Therefore,
\begin{eqnarray*}
\frac{d}{d t}\biggl|_{t=0}\Phi_{g(t),X}^*(\omega_t(h_1))
          &=&R_f\biggl((\tau_{f}^*)^{-1}[-\frac{hh_1\Omega}{f^2}\biggl|_{U_\mathcal{N}}]\biggr)\\
          &=&R_f([-\frac{hh_1\Omega}{f^2}]).\\
\end{eqnarray*}
Now the proof of Lemma \ref{lem: dp} is complete.

\begin{rk}
To prove Theorem \ref{main cor}, it is essential for us to obtain a diagram like \eqref{eq: square}. In fact, when Y. Zhao \cite{Zhao} proves the infinitesimal Torelli theorem for nodal surfaces, he uses such a diagram implicitly; however, he does not give any proofs. We believe that a detailed proof is indeed needed and this is a special reason why our discussions above always include the case $n=3$.
\end{rk}

\end{document}